\documentclass[a4paper,reqno,11pt,oneside]{amsart}
\usepackage{amssymb,amsmath,amsthm,amsfonts}
\usepackage[bitstream-charter]{mathdesign}
\usepackage{fix-cm}
\usepackage[utf8]{inputenc} 
\usepackage[T1]{fontenc} 
\usepackage[english]{babel} 
\usepackage{nicefrac} 
\usepackage{hyperref}
\usepackage{geometry}
\usepackage{graphicx}
\usepackage{braket}
\usepackage{subfig}
\usepackage{csquotes}
\usepackage{color}
\usepackage[nobysame]{amsrefs}

\providecommand{\abs}[1]{\left|#1\right|}
\providecommand{\norm}[1]{\left \| #1\right \|}

\newtheorem{proposition}{Proposition}[section]
\newtheorem{theorem}[proposition]{Theorem}
\newtheorem{corollary}[proposition]{Corollary}
\newtheorem{lemma}[proposition]{Lemma}

\newtheorem{remark}[proposition]{Remark}

\numberwithin{equation}{section}

\newcommand{\R}{\mathbb{R}}

\newcommand{\intr}{\int_{\R^{N}}}

\newcommand{\eps}{\varepsilon}

\title[Positive solutions to nonlinear Choquard equations]
{Positive bound states to nonlinear Choquard  equations 
in the presence of nonsymmetric potentials}
\date{}
\author{Liliane Maia}
\address[L. Maia]{Departamento de Matem\'atica, 
Universidade de Bras\'\i lia, 70910-900 Bras\'\i lia, Brazil}
\email[L. Maia]{lilimaia@unb.br}
\author{Benedetta Pellacci}
\address[B. Pellacci]{Dipartimento di Matematica e Fisica,
Universit\`a della Campania  ``Luigi Vanvitelli'',  via A. Lincoln 5, 81100
Caserta, Italy.}
\email[B. Pellacci]{benedetta.pellacci@unicampania.it}
\author{Delia Schiera}
\address[D. Schiera]{Dipartimento di Matematica e Fisica,
Universit\`a della Campania  ``Luigi Vanvitelli'',  via A. Lincoln 5, 81100
Caserta, Italy.}
\email[D. Schiera]{delia.schiera@unicampania.it}

\subjclass[2010]{ 45K05, 35Q55, 35J91, 35J20. }
\keywords{
 Choquard equations, Nonlocal nonlinearities, Positive solutions.}
\thanks{Research partially supported by:   PRIN-2017-JPCAPN Grant: ``Equazioni 
differenziali alle derivate parziali non lineari'',
by project Vain-Hopes within the program VALERE: VAnviteLli pEr la RicErca and by the INdAM-GNAMPA group. 
L.\ Maia was partially supported by FAPDF, CAPES, and CNPq grant 309866/2020-0.}

\begin{document}
\maketitle

\begin{abstract}
The  existence of a positive solution to a class of Choquard equations with potential going at a positive limit at infinity possibly from above or oscillating is proved.
Our results  include the physical case and do not require any symmetry assumptions on the potential.
\end{abstract}
\noindent

\section{Introduction}
In this paper we will prove the existence of a positive solution of the following problem
\begin{equation}\label{Choquardeq}\tag{$P_V$}
\begin{cases}
-\Delta u + V(x) u =  (I_{\alpha} \ast u  ^{2})u  & \text{in } \mathbb{R}^N, 
 \\
\qquad u \in H^{1}(\mathbb{R}^N), &
 \end{cases}
\end{equation}
where 
\begin{equation}\label{range2}
\alpha \in ((N-4)^{+},N-1]
 \end{equation}
 and $I_{\alpha}$ represents the Riesz operator of order $\alpha$, defined for each point $x \in \mathbb{R}^N \setminus \{ 0 \}$ by
\[ I_{\alpha}(x)=\frac{A_{\alpha}}{\abs{x}^{N-\alpha}}, \quad \text{where } \quad A_{\alpha}=\frac{\Gamma(\frac{N-\alpha}{2})}{\Gamma(\alpha/2) 2^{\alpha} \pi^{N/2}},\]
and  the potential $V$ is such that  
\begin{equation}\label{V1} 
V \in C^0, \inf_{x \in\mathbb{R}^N} V(x) >0, \text{ and  } \lim_{\abs{x} \to \infty} V(x) =V_{\infty}\in (0,+\infty).
\end{equation}

This equation  appears in the context of various  physical models and
we refer to \cite[Section 2]{MorozVanSchaftJFPTA} for an extensive introduction
on the physical context. 

When $V(x)\equiv V_{\infty}$ \eqref {Choquardeq} reduces to the
autonomous problem
\begin{equation}\label{Choqlimit}\tag{$P_{\infty}$}
-\Delta u + V_{\infty} u =  (I_{\alpha} \ast \left| u \right| ^{2}) u \quad \text{in } \mathbb{R}^N,
\end{equation}
and in the physical case, $N=3,\, \alpha=2$, the first  result goes back to Lieb 
(\cite{Lieb}) who proves the existence of  a normalized solution, corresponding to  
the unique minimum point of the energy functional on the $L^{2}(\R^{3})$ sphere.
The existence of infinitely many radial symmetric solutions has been obtained 
by Lions in \cite{Lions} again for $N=3$ and $\alpha=2$.
These existence results have been extended by Moroz and Van Shaftinghen  
(\cite{MorozVanSchaftJFA})
to different exponents $\alpha$, different dimensions $N$ and even to  powers different from the square on the function $u$, 
starting to the fact that, by the Hardy--Littlewood--Sobolev inequality, the right hand side of \eqref{Choqlimit} (and of \eqref{Choquardeq}) is well defined on $H^1(\mathbb{R}^N)$ when
\begin{equation}\label{rangep}
\frac{N-2}{N+\alpha} < \frac{1}{p} < \frac{N}{N+\alpha}.
 \end{equation}
More precisely, in \cite{MorozVanSchaftJFA} it is proved that \eqref{Choqlimit} 
has a positive radially symmetric least action solution $\omega 
\in C^2(\mathbb{R}^N)$.  The question of the validity of the uniqueness of positive 
solution has been already addressed by Lieb in \cite{Lieb} who  proves it
for normalized solutions; this result has been extended to any positive solution
by Ma and Zhao (\cite{MaZhao}) who prove   that
there exists  a unique positive solution of Problem \eqref{Choqlimit} for
$\alpha=2$ and $N=3$; this results has been extended for $N=4,5$ in
\cite{wangyi}. Let us also mention that the uniqueness property of the least action 
solution  has been  proved  in \cite{Xiang} for  powers $p\neq 2$ belonging in a 
suitable range.

Precise decay estimates for $\omega$ are proved in \cite{MorozVanSchaftJFA,MorozVanSchaftJDE};
in particular the decay turns out to be exponential in our  case for $\alpha<N-1$ and a polynomial perturbation of an exponential decay for $\alpha=N-1$  (see for more details
Theorem \ref{thm:decay} in Section \ref{setting}).

Coming back to the non-autonomuous Problem \eqref{Choquardeq}, when $V(x)\leq V_{\infty}$ the existence of a least action  solution is due to 
\cite{Lions2} (see also \cite{MorozVanSchaftJFPTA, VanShaftXia}) and can be obtained by minimizing the associated action functional
\[ 
\mathcal{I}_V(u)=\frac{1}{2} \int_{\mathbb{R}^N}(\abs{\nabla u}^2 + V(x)u^2) - \frac{1}{4} \int_{\mathbb{R}^N} (I_{\alpha} \ast u^2) u^2 . 
\]
on the Nehari manifold
\[
\mathcal{N}_V =\left\{ u \in H^1(\R^{N}) \setminus \{ 0 \} : 
\langle \mathcal{I}'_V(u), u \rangle=0\right\}.
\]
But, when $V(x)$ approaches $V_{\infty}$ from above or oscillating, one is 
forced to look for higher action level solutions, and as a consequence, a deeper 
comprehension of  the  possible lack of compactness of a bounded Palais-Smale 
sequence is needed.

To  this aim, a well-known tool is the so-called Splitting Lemma 
(see \cite{BenciCerami, Struwe}), whose application requests the uniqueness of positive 
solutions of \eqref{Choqlimit}, which, as above observed, is known for
$p=\alpha=2$ and $N=3,4,5$.

The use of the Splitting Lemma allows to detect an action level's interval where
compactness is recovered, so that the existence of a critical point can be
obtained by  constructing a minimax level in this interval.
In this construction, a precise knowledge of the decay of $\omega$ is needed, 
and it is  crucial a meticulous comparison between the asymptotical decay of 
the solution of the limit problem and the  decay of the potential acting
in the problem.

Following this path we will prove the following result
\begin{theorem}\label{thm:uniq}
Assume $\alpha=2$,  $N=3, 4, 5$ and that \eqref{V1} holds. We also assume that the potential $V(x)$ satisfies  
 \begin{equation}\label{Vdecay}
V(x) \le V_{\infty} + A_0 \abs{x}^{\sigma}e^{-\beta\abs{x}} , \quad \text{with } A_0 >0,
\forall  x \in \R^N,
\end{equation}
and the exponent $\sigma$ is such that
\begin{equation}\label{iposigma}
\begin{cases}
\sigma\in \R & \text{if }  \beta > 2 \sqrt{ V_{\infty}}
\\
\sigma <- 2  & \text{if }
\beta = 2 \sqrt{ V_{\infty}} \text{ and $N=5$}, 
\\
\sigma <- \frac32  & \text{if }
\beta = 2 \sqrt{ V_{\infty}} \text{ and $N=4$}, 
\\
\sigma <-1   & \text{if }
\beta = 2 \sqrt{ V_{\infty}} \text{ and $N=3$}.
\end{cases}
\end{equation}
 Then there exists a positive solution to \eqref{Choquardeq}. 
\end{theorem}
Theorem \ref{thm:uniq} will be a direct consequence of an abstract result stated in Section
\ref{setting} (Theorem \ref{thm:abstract}). 
Let us point out that in Theorem \ref{thm:uniq} it is admitted the possibility 
that $V(x)$ approaches $V_{\infty }$ from above or oscillating.
Moreover, we believe that the decay assumptions  on $V(x)$ 
are  optimal in this type of argument and are naturally strictly related to the decay estimates of
$\omega$ which varies when $N=3$ and $N=4, 5$ (see for more details
Theorem \ref{thm:decay}).  

Other strategies to find nontrivial solutions, 
avoiding the use of the uniqueness properties  of the limit Problem \eqref{Choqlimit}, have been implemented in the last years (see \cite{clasal, CingolaniClappSecchi, GhimentiMorozVanSchaft, GhimentiVanSchaft, Alves} and the references therein)
In particular, Clapp and Salazar (\cite{clasal}) take advantage of the use of symmetries to increase the minimum action level and to show the existence of a positive (or even sign-changing) solution for potentials enjoying the same symmetries. This has been possible requiring  an enough high level of symmetries
in order to construct a minimax level into the range of  compactness.
Moreover, the decay of the potential is assumed to be of exponential 
type with a negative exponent that naturally depends on the symmetries and it is 
sufficiently large in modulus. The use of symmetries has been also adopted by
Cingolani, Clapp and Secchi in  \cite{CingolaniClappSecchi} to obtain existence results 
for a class of magnetic nonlinear Choquard equations.

Here, taking advantage of the uniqueness property of the limit Problem 
\eqref{Choqlimit}, we will not exploit any symmetries' action, so that our 
existence result does not require any invariance property of the potential. 
In addition, our decay assumptions on $V(x)$ include the possibility
that the exponential part in the decay is  equal to the exponential 
decay of $\omega^{2}$.

Let us conclude the introduction mentioning the existence result of positive 
solutions contained in  \cite{WangQuXiao}, where a Choquard equation with 
competing potentials is studied in the case $p=\alpha=2$, $N=3$ under some 
stronger  decay assumptions on the potentials than the one assumed here.

 \medskip
This paper is organized as follows: in Section \ref{setting} we give the variational setting of the problem and some preliminary results, whereas in Section \ref{estimates} we  get the fundamental asymptotic estimates we need in the proof of the main results. 
The proof of Theorems \ref{thm:abstract} and \ref{thm:uniq} is performed in
Section \ref{final}.
\section{Setting of the problem and preliminaries}\label{setting}

We will work in the functional space $H^1(\R^{N})$
 endowed, thanks to \eqref{V1}, with the scalar product and norm, equivalent to
 the usual one
\[ 
( u, v )_V = \int_{\mathbb{R}^N} (\nabla u\cdot \nabla v + V(x) uv), \quad \quad \norm{u}^2_V=\int_{\mathbb{R}^N}(\abs{\nabla u}^2 + V(x) u^2 ). 
\]
Every  solution to \eqref{Choquardeq} is a critical point of the action
functional $\mathcal{I}_V: H^{1}(\R^{N})\mapsto \R$ defined by
\[ 
\mathcal{I}_V(u)=\frac{1}{2} \int_{\mathbb{R}^N}(\abs{\nabla u}^2 + V(x)u^2) - \frac{1}{4} \int_{\mathbb{R}^N} (I_{\alpha} \ast u^2) u^2 . 
\]
where $\alpha $ satisfies \eqref{range2}.

Hypothesis \eqref{range2} and Hardy-Littlewood-Sobolev inequality
imply that $\mathcal{I}_V$ is a $C^{1}$ functional on $H^1(\R^{N})$, (see \cite[Proposition 3.1]{MorozVanSchaftJFPTA}), so that we can define
\begin{equation}\label{defnehari}
\mathcal{N}_V =\left\{ u \in H^1(\R^{N}) \setminus \{ 0 \} : 
\langle \mathcal{I}'_V(u), u \rangle=0\right\},\qquad 
 c_V =\inf_{u \in \mathcal{N}_V} \mathcal{I}_V(u).
 \end{equation}
In an analogous way, we can define $\mathcal{I}_{\infty}: H^{1}(\R^{N})\mapsto \R$ by
\[
\mathcal{I}_{\infty}(v)=\frac{1}{2} \int_{\mathbb{R}^N}(\abs{\nabla u}^2 + V_{\infty}u^2) - \frac{1}{4} \int_{\mathbb{R}^N} (I_{\alpha} \ast  u^2) u^2 ,
\]
where  $H^{1}(\R^{N})$ is endowed with the norm and the scalar product in
\begin{equation}\label{normainfty}
( u, v )  = \int_{\mathbb{R}^N} (\nabla u \nabla v + V_{\infty} uv), \quad \quad \norm{u}^2 =\int_{\mathbb{R}^N}(\abs{\nabla u}^2 +V_{\infty}  u^2 ). 
\end{equation}
Accordingly $\mathcal{N}_\infty (u)$ and  $c_{\infty}$ are defined as follows
\begin{equation}\label{eq:nehariinf}
\mathcal{N}_\infty =\left\{ u \in H^1(\R^{N}) \setminus \{ 0 \} : 
\langle \mathcal{I}'_\infty(u), u \rangle=0\right\},\qquad 
 c_\infty =\inf_{u \in \mathcal{N}_\infty} \mathcal{I}_\infty(u).
\end{equation}

As already mentioned in the Introduction,
the existence of a least action solution to \eqref{Choqlimit} is proved, under 
assumption \eqref{range2}, in  Theorem 3.2 in \cite{MorozVanSchaftJFPTA}. 
Moreover, weak solutions are classical, and, up to translation and inversion of the 
sign, positive and radially symmetric, see \cite{MorozVanSchaftJFA, Lieb}. 
In addition, precise decay asymptotic  for solutions to \eqref{Choqlimit} are 
given in Propositions 6.3, 6.5 and Remark 6.1 in \cite{MorozVanSchaftJFA}, 
(see  also \cite{MorozVanSchaftJDE}),  and they are summarized in the following result.
\begin{theorem}[Theorem 4 pg.157, Remark 6.1 pg.177 in \cite{MorozVanSchaftJFA}]\label{thm:decay}

Let $\omega$ a least action solution to \eqref{Choqlimit}.
Then the following asymptotic estimates hold.
\\
If $(N-4)^{+}<\alpha<N-1$ it results
\begin{equation}\label{decay2} 
\omega(x)= (c+o(1))\abs{x}^{-\frac{N-1}{2}}
e^{-\sqrt{ V_{\infty}}\abs{x}} \quad \text{with $c>0$  and as } \abs{x} \to \infty. 
\end{equation} 
If  $\alpha=N-1$, $\omega$ decays at infinity as follows
\begin{equation}\label{decaypertpol} 
 \omega(x) = (c+o(1)) \abs{x}^{-\frac{N-1- \nu\sqrt{V_{\infty}}}{2} } e^{-\sqrt{V_{\infty}} \abs{x}}, \quad \text{with $c>0$  and as } \abs{x} \to \infty. 
\end{equation}
where $\nu$ is a positive constant depending on the $L^{2}(\R^{N})$ norm of $\omega$
(see \eqref{Q} below).
\end{theorem}
The above result shows that the interaction of  the Riesz potential and the nonlinearity affects in a substantial way the decay of the least action solutions.
In our context we can see different perturbations on the decay of $\omega$ depending on $\alpha$. 
In general, it holds
\begin{equation}\label{decay3} 
\omega(x)= (c+o(1)) \frac{e^{- \sqrt{V_{\infty}}Q(\abs{x})}}{\abs{x}^{\frac{N-1}{2}}} 
\quad \text{ as } \abs{x} \to \infty, 
\end{equation}
where 
\begin{equation}\label{Q} Q(t)= \int_{\nu}^{t} \sqrt{1 - \frac{\nu^{N-\alpha}}{s^{N-\alpha}}}  \, ds, 
\qquad  \nu^{N-\alpha} = \frac{1}{V_{\infty}} \frac{\Gamma(\frac{N-\alpha}{2})}{\Gamma(\frac{\alpha}{2}) \pi^{N/2} 2^{\alpha} } \int_{\mathbb{R}^N} \abs{\omega}^2. 
\end{equation}
Notice that $\nu$ does not actually depend on the choice of $\omega$, as $\norm{\omega}_2^2$ is invariant among least action solutions (see \cite{MorozVanSchaftJFA}).
Nevertheless, when $\alpha < N-1$, a careful analysis of the function $Q$ shows that \eqref{decay2}  holds, whereas if $\alpha=N-1$ (which includes the physical case $N=3$, $\alpha=2$), a perturbation in the polynomial part occurs and one gets \eqref{decaypertpol}; if $N-1< \alpha < N$, then more involved perturbations appear, as a result of the Taylor expansion for the square root.   
However, in our Theorems \ref{thm:uniq}, \ref{thm:abstract} we will take into consideration the asymptotical decay
given in \eqref{decay2} for the cases $\alpha<N-1$ and  
in \eqref{decaypertpol} for $\alpha=N-1$, which includes
the physical case $\alpha=2$ and $N=3$.

We will prove Theorem \ref{thm:uniq} as a consequence of the  following result.
\begin{theorem}\label{thm:abstract}
Let  $N\geq 2$,  $\alpha\in ((N-4)^{+},N-1]$, and suppose that \eqref{V1} holds. We also assume that the potential $V(x)$ satisfies  
 \begin{equation}\label{Vdecaygen}
V(x) \le V_{\infty} + A_0 \abs{x}^{\sigma}e^{-\beta\abs{x}} , \quad \text{with } A_0 >0,
\forall  x \in \R^N,
\end{equation}
and the exponent $\sigma$ is such that
\begin{equation}\label{ipo:sigma}
\begin{cases}
\sigma\in \R & \text{if }  \beta > 2 \sqrt{ V_{\infty}}
\\
\sigma <\min\left\{-1,-\frac{N-1}2\right\} & \text{if $\beta=2\sqrt{V_{\infty}}$ and 
$\alpha<N-1$,}
\\
\sigma <\min\left\{-1,-\frac{N-1}2+\nu\sqrt{V_{\infty}}\right\} & \text{if $\beta=2\sqrt{V_{\infty}}$ and 
$\alpha=N-1$,}
\end{cases}
\end{equation}
where $\nu$ is introduced in \eqref{decaypertpol}.
Then, if the limit Problem $(P_{\infty})$ has a unique positive solution, 
there exists a positive solution of $(P_{V})$.
\end{theorem}
\begin{remark}
Let us observe that, for $\alpha<N-1$ 
  the hypothesis \eqref{ipo:sigma}   requires
$\sigma<-1$ when $N=2$ and $\sigma <-(N-1)/2$ when $N\geq 3$.
\end{remark}
{\begin{remark} 
Let  us  note that as observed in \cite{GhimentiMorozVanSchaft}  
the hypothesis  $\alpha\in ((N-4)^{+},N)$ is fundamental in order to have
the convolution term well defined in $H^{1}(\R^{N})$ as a consequence of the Hardy-Littlewood-Sobolev inequality.
Notice that when $\alpha=2$ this amounts to consider
$N\leq 5$, so that this upper bound on the dimension is needed
from the beginning, in order to have the convolution term well-defined.
\end{remark}
{\begin{remark}
Theorem \ref{thm:abstract} does not include the case $\beta=2\sqrt{V_{\infty}}$
and $\alpha \in (N-1,N)$. In this range  the decay of $\omega$ changes 
(\cite{MorozVanSchaftJFA}, \cite{MorozVanSchaftJDE}). 
An analogous result can be obtained,  also in the case $p=2$, $\alpha > N-1 $. 
But, the principal tool in order to obtain the decay estimates (see Lemma 
\ref{ACR}) cannot be directly applied; the interested reader can see
\cite{MaPeSc} where we prove an extension of Lemma \ref{ACR} to handle the 
case $p=2,\, \alpha\in (N-1,N-\frac12).$
\end{remark}

Let us conclude this section by recalling the following  decay information concerning  the convolution term.
\begin{lemma}\label{le:convo}
 Let $h \ge 0$, $h\in L^{\infty}$ such that 
\begin{equation}\label{eq:bound}
 \sup_{\mathbb{R}^N} h(x) (1+\abs{x})^{s} < + \infty 
 \end{equation}
for some $s >N$. Then
\[ I_{\alpha} \ast h(x) = I_{\alpha}(x) \norm{h}_1 (1+ o(1)), \qquad \text{as $\abs{x}\to \infty$.}
\]
\end{lemma}
\begin{proof}
The  conclusion follows immediately from Lemma 6.2 in \cite{MorozVanSchaftJFA}.
\end{proof}
As an immediate consequence of Lemma \ref{le:convo}, we get the following asymptotical decay of the convolution term  
\begin{equation}\label{decayconv}
I_{\alpha} \ast \omega^2(x) = I_{\alpha}(x) \norm{\omega}_2 (1+ o(1)), \qquad \text{as $\abs{x}\to \infty$.}
\end{equation}
Indeed, taking $h=\omega^{2}$ one immediately has that \eqref{eq:bound} is satisfied for every $s$.

\section{Asymptotic estimates}\label{estimates}
In this section we will obtain all the   asymptotic estimates 
 needed in proving our main results.
We will first introduce the threshold that will guide our study and we will establish
its decay. Then,  the asymptotical decay of the integral term involving the potential will be studied, and at last we will deal with the nonlinearity term. 

Let us precise that with the expression $f\sim g$ as $|x|\to \infty$ we mean that 
the quotient $f/g\to l\in(0,+\infty)$ as $|x|\to \infty$.
The  following Lemma will be repeatedly exploited.
\begin{lemma}[Lemma 3.7 in \cite{AmbrosettiColoradoRuiz}]\label{ACR}
Let $u, v : \mathbb{R}^N \to \mathbb{R}$ be two positive continuous radial functions such that
\[ 
u(x) \sim \abs{x}^a e^{-b\abs{x}}, \quad v(x) \sim \abs{x}^{a'} e^{-b'\abs{x}}  
\qquad \text{as $|x| \to \infty$,}
\]
 where $a, a' \in \mathbb{R}$, and $b, b' >0$. Let $\xi \in \mathbb{R}^N $ tend to infinity. We denote $u_{\xi}(x)=u(x-\xi)$. Then the following asymptotic estimates hold
\begin{itemize}
\item[(i)] If $b < b'$,
\[ \int_{\mathbb{R}^N} u_{\xi} v\sim e^{-b\abs{\xi}} \abs{\xi}^{a}. \]
A similar expression holds if $b > b'$, by replacing $a$ and $b$ with $a'$ and $b'$. 
\item[(ii)] If $b=b'$, suppose that $a \ge a'$. Then
\[ \int_{\mathbb{R}^N} u_{\xi} v\sim
\begin{cases}
e^{-b\abs{\xi}} \abs{\xi}^{a+a'+\frac{N+1}{2}} & \text{ if } a' > -\frac{N+1}{2},\\
e^{-b\abs{\xi}} \abs{\xi}^{a} \log \xi & \text{ if } a' = -\frac{N+1}{2},\\
e^{-b\abs{\xi}} \abs{\xi}^{a}& \text{ if } a' < -\frac{N+1}{2}.
\end{cases} \]
\end{itemize}
\end{lemma}
Let  $z_{1}\in \R^{N}$ with $\abs{z_{1}}=1$ and $z_{2}\in \partial B_{2}(z_{1})$,
we denote  with  $ \omega_{i, R}(x)$ a positive  solution of \eqref{Choqlimit} achieving $c_{\infty}$ (see \eqref{eq:nehariinf})
of the limit problem translated in $Rz_{i}$,
namely
\begin{equation}\label{omegaR}
 \omega_{i, R}(x)=\omega(x- Rz_i ). 
\end{equation}

Moreover, the threshold  guiding  all the asymptotic estimates is 
\begin{equation}\label{defeps}
 \varepsilon_R= \int_{\mathbb{R}^N} (I_{\alpha} \ast \omega_{1, R}^2) \omega_{1, R}  \omega_{2, R}=\int_{\mathbb{R}^N} (I_{\alpha} \ast \omega_{2, R}^2) \omega_{2, R}  \omega_{1, R}. 
 \end{equation}

The precise decay of $\eps_{R}$
is obtained in the following lemma.

\begin{lemma}\label{le:estepsnew}
Let $\alpha\in ((N-4)^{+},N-1]$.   Then, for $R$ large enough,  the following conclusions hold.
\begin{equation}\label{eq:esteps} 
\varepsilon_R  \sim
\begin{cases} 
e^{- 2 \sqrt{V_{\infty}} R}R^{-\frac{N-1}{2}} ,  & \text{if }\alpha< N-1.
\\
e^{-2  \sqrt{V_{\infty}} R} R^{ - \frac{N-1}2+ \nu  \sqrt{V_{\infty}}} ,& \text{if }\alpha= N-1.
\end{cases}
\end{equation}
where $\nu$ is introduced in \eqref{Q}.
 \end{lemma}
 \begin{remark}\label{rem:p=2}
Notice that, for any $\alpha \in (0, N)$, it results 
\[
\varepsilon_R  \ge C R^{-\frac{N-1}{2}} e^{-   \sqrt{V_{\infty}} Q(2R   )} \]
where $Q$ is introduced in \eqref{Q}. 
Indeed,  one has
\[ \inf_{x \in B_1(0)} I_{\alpha} \ast \omega^2(x) \ge  \inf_{x \in B_1(0)} \frac{A_{\alpha}}{R_0^{N-\alpha}} \int_{B_{R_0}(x)} \omega^2(y) \, dy 
\ge 
A_{\alpha} \frac{\abs{B_{R_0}(0)}}{R_0^{N-\alpha}} \min_{y \in B_{R_0+1}(0)} \omega^2(y) 
\ge C >0.\]
Therefore, denoting with $C$ possibly different constants and taking into
account the general decay given in \eqref{decay3} and recalling that $Q$
is monotone increasing for $t>\nu$, one gets
\[\begin{split}
\varepsilon_R  
&\ge 
\int_{B_1(Rz_1)} (I_{\alpha} \ast \omega_{i, R}^2) \omega_{i, R} \omega_{j, R} 
= \int_{B_1(0)} (I_{\alpha} \ast \omega^2(x)) \omega(x) \omega(x- R(z_2 -z_1)) \, dx \\
&\ge  \inf_{x \in B_1(0)} ( I_{\alpha} \ast \omega^2(x)  \omega(x) ) \int_{B_1(0)} \omega(x- R(z_2-z_1)) \, dx \\
&\ge  C \int_{B_1(0)} e^{-  \sqrt{V_{\infty}}Q(\abs{x - R(z_1- z_2)})} \frac{1}{(1+\abs{x-R(z_2-z_1)})^{\frac{N-1}{2}}} 
\\
&
\ge C R^{-\frac{N-1}{2}} e^{-  \sqrt{V_{\infty}} Q(1+2R )},
 \end{split}\]
for $R$ sufficiently large. Notice that 
\[
\begin{split} 
Q(1+2R ) -Q(2R ) &= \int_{\nu}^{1+2R }  \sqrt{ 1 - \frac{\nu^{N-\alpha}}{s^{N-\alpha}}}  \, ds - \int_{\nu}^{2R }  \sqrt{ 1- \frac{\nu^{N-\alpha}}{s^{N-\alpha}}}  \, ds 
\\
&= \int_0^1 \sqrt{ 1 - \frac{\nu^{N-\alpha}}{(t+2R)^{N-\alpha}}}  \, dt  \le c, \quad \text{as $R \to \infty$.}
\end{split}
\] 
In addition $Q(2R)\leq 2R-\nu$
so that $\eps_{R}\geq C_{0}R^{-\frac{N-1}2}e^{-2R\sqrt{V_{\infty}}}$, which shows that
\eqref{eq:esteps} is optimal for $\alpha<N-1$. 
On the other hand, when $\alpha=N-1$,
this estimate from below is consistent with estimates obtained in 
Lemma \ref{le:estepsnew} and also with estimates in \cite{clasal}, however it is far from being sharp.
\end{remark}

\begin{proof}[Proof of Lemma \ref{le:estepsnew}]
Let us first observe that, performing a change of variable
\begin{equation}\label{cambiovar}
\eps_{R}=\int_{\R^N} (I_{\alpha}\ast\omega^{2})(x)\omega (x)\omega(x-R(z_{1}-z_{2}))dx.
\end{equation}
We are going to apply Lemma \ref{ACR} with 
\begin{equation}\label{choiceAR}
v=I_{\alpha} \ast \omega^2 \omega , \qquad u=\omega, \qquad \xi=R(z_1 - 
z_2), \quad \text{and $|\xi|=2R$},
\end{equation}
with the exponents
\[
b=b'=\sqrt{V_{\infty}}, \quad a=-\frac{N-1}{2},    \quad a'=a-N+\alpha<a
\]
and $a'<-(N+1)/2$ iff $\alpha<N-1$, so that in this case  we get the first information in \eqref{eq:esteps}, while when  $\alpha=N-1$ we have to consider \eqref{decaypertpol}
\[
b=b'=\sqrt{V_{\infty}}\quad a=-\frac{N-1}{2}+\frac{\nu\sqrt{V_{\infty}}}2,  \quad a'=a-N+\alpha=a-1<a 
\]
and now $a'>-(N+1)/2$ as $a>-(N-1)/2$ so that the second information in \eqref{eq:esteps}  follows.
\end{proof}
Let us now prove the asymptotic estimates on the term with $V(x)$ that  will be used in
the following.
\begin{lemma}\label{estimatesV}
Let  $N\geq 2$, and $\alpha \in ((N-4)^{+}, N-1]$
and assume \eqref{V1}, \eqref{Vdecaygen}, \eqref{ipo:sigma}. 
Then, for $R$ large enough, it results
\[ \mathcal{A}_V:=\int_{\mathbb{R}^N} (V(x) - V_{\infty})\left(\omega_{i, R}\right)^2 \leq  o(\varepsilon_R), \quad \text{for $i=1, 2$.}
 \]
\end{lemma}

\begin{remark}
In the proof of this Lemma we will exploit Lemma \ref{ACR}.
Notice that, in order to do this, we will first make use of \eqref{Vdecaygen} 
which gives an upper bound on $V$ by a positive radial function.
As a consequence, we will get the conclusion, even if $V(x)-V_{\infty}$ is not radial, nor positive.
\end{remark}

\begin{proof} 

We want now to apply Lemma \ref{ACR}. However, this only applies to radial functions, hence we preliminary notice that 
\[ \int_{\mathbb{R}^N} (V(x) - V_{\infty}) \omega_{i, R}^2 \le  C \int_{\mathbb{R}^N}  (\abs{x}+1)^\sigma e^{-\beta \abs{x}} \omega_{i, R}^2. \]
Take
\[ u=\omega^2, \quad v=  (\abs{x}+1)^\sigma e^{-\beta \abs{x}}, \quad
\text{and $\xi=Rz_i$. }
\]
Notice that 
\[\begin{split}
u &\le Ce^{-2\sqrt{V_{\infty}}\abs{x}} \abs{x}^{-N+1},\qquad\quad\quad \text{  for $\alpha< N-1$}
\\
u &\le Ce^{-2\sqrt{V_{\infty}}\abs{x}} \abs{x}^{-N+1+\nu\sqrt{V_{\infty}}}, \quad\,\;\text{for $\alpha=N-1$, }
\end{split}\]
with $C$ a suitable positive constant.
Then, in the case $\beta>2\sqrt{V_{\infty}}$ Lemma \ref{ACR} immediately implies
the conclusion.
Let us now deal with  the  case $\beta=2\sqrt{V_{\infty}}$ considering first $N=2$
and $\alpha<N-1$, i.e. $\alpha\in (0,1)$.
we  again exploit Lemma \ref{ACR} with (recalling  \eqref{ipo:sigma})
\[
 b=b'=2\sqrt{V_{\infty}}\quad a=-1,  \quad a'=\sigma< -1 ;
\]
then one has $a>a'$, and we can suppose without loss of generality that
$a'>-3/2$, so that Lemma \ref{ACR} yields the conclusion as $\sigma<-1$.
If $N=2,\,\alpha=1$, we have
\[
 b=b'=2\sqrt{V_{\infty}}\quad a=-1+\nu\sqrt{V_{\infty}},  \quad a'=\sigma<-1;
\]
so that $a>a'$ and again we can suppose that $a'>-3/2$, and we reach the conclusion taking into account \eqref{eq:esteps} .
When $N\geq 3$ and $\alpha <N-1$ we can use Lemma \ref{ACR} with
\[
b=b'=2\sqrt{V_{\infty}},\quad a=1-N, \quad   a'=\sigma< -\frac{N-1}2 
\]
and $a'>a$ and $a<-(N+1)/2$ for every $N>3$ and $a=-(N+1)/2$ if $N=3$.
In both cases, we get $a'$ as the exponent of the polynomial part
and the conclusion  follows as $a'=\sigma<-\frac{N-1}2 $.
As a last case, let us consider $N\geq 3$ and $\alpha=N-1$. Take 
\[
b=b'=2\sqrt{V_{\infty}},\quad a=1-N+\nu\sqrt{V_{\infty}}, \quad   
a'=\sigma< \min\left\{-1,-\frac{N-1}2 +\nu\sqrt{V_{\infty}}\right\}.
\]
If $\nu\sqrt{V_{\infty}}\leq \frac{N-3}2$ we have $a\leq a'$ and  $a\leq -(N+1)/2$
so that the exponent in the polynomial term will be  $a'=\sigma$ yielding 
the conclusion as $\sigma <-\frac{N-1}2+\nu\sqrt{V_{\infty}}$.
In the case in which $\frac{N-3}2<\nu\sqrt{V_{\infty}}\leq N-2$ one has 
$a\leq a'$ and  $a>-(N+1)/2$ so that the exponent in the polynomial term will be  $a'+a+\frac{N+1}2=\sigma-\frac{N-3}2+\nu\sqrt{V_{\infty}}$ implying again 
the conclusion thanks to \eqref{ipo:sigma}.
Finally, when $\nu\sqrt{V_{\infty}}>N-2$ it results $a>a'$, and as
$-(N+1)/2<-(N-1)/2+\nu\sqrt{V_{\infty}}$ for every $N$ we can suppose
w.l.o.g. that $\sigma >-(N+1)/2$ so that the exponent in the polynomial
term will be $a+\sigma+\frac{N+1}2$ which again gives a decay faster than
the one of $\eps_{R}$.
\end{proof}
Let us conclude this section by studying the nonlinearity term.
\begin{proposition}\label{prop:somma}
Given $s,\,t \in (0,+\infty)$, it results  
\[ \begin{split}
 \int_{\mathbb{R}^N} \left(I_{\alpha} \ast (s\omega_{1,R}+t\omega_{2,R})^2 \right)
 (s\omega_{1,R}+t\omega_{2,R})^2
&-  s^{4}\int_{\mathbb{R}^N} (I_{\alpha} \ast  \omega_{1, R}^2)  \omega_{1, R}^2 -
t^{4}\int_{\mathbb{R}^N} (I_{\alpha} \ast  \omega_{2, R}^2)  \omega_{2, R}^2 
\\
&\geq  4st(s^{2}+t^{2}) \eps_{R}.
\end{split} \]
where $\eps_{R}$ is defined in \eqref{defeps}.
\end{proposition}

\begin{proof}
Direct computations show
\[\begin{split}
 \int_{\mathbb{R}^N} \left(I_{\alpha} \ast (s\omega_{1,R}+t\omega_{2,R})^2 \right)
 (s\omega_{1,R}+t\omega_{2,R})^2
\geq&
s^{4}\int_{\mathbb{R}^N} (I_{\alpha} \ast  \omega_{1, R}^2)  \omega_{1, R}^2 +
t^{4}\int_{\mathbb{R}^N} (I_{\alpha} \ast  \omega_{2, R}^2)  \omega_{2, R}^2 
\\
+4st (s^{2}+t^{2})\eps_{R}
\end{split}\]
where we have used that
\[
\int_{\R^{N}}\left(I_{\alpha}\ast \omega_{i,R}^{2}\right)\omega_{j,R}^{2}\geq 0 \qquad \text{for every $i,\,j=1,2$ with $i\neq  j$}
\]
and that
\[\begin{split}
\int_{\R^{N}}\left(I_{\alpha}\ast \omega_{i,R}\omega_{j,R}\right)\omega^{2}_{i,R}
&=
\int_{\R^{N}}\omega^{2}(\theta-Rz_{i}) d\theta
\int_{\R^{N}}\frac{\omega(x-Rz_{i})\omega(x-Rz_{j})}{|x-\theta|^{N-\alpha}}dx
\\
&=\int_{\R^{N}}
\omega(x-Rz_{i})\omega(x-Rz_{j}) dx
\int_{\R^{N}}\frac{\omega^{2}(\theta-Rz_{i}) }{|x-\theta|^{N-\alpha}}d\theta
\\
&=\int_{\R^{N}}\left(I_{\alpha}\ast \omega_{i,R}^{2}\right)\omega_{j,R}\omega_{i,R}=\eps_{R}.
\end{split}\]

\end{proof}


\section{Proof of Theorems \ref{thm:abstract} and \ref{thm:uniq}}
\label{final}

Let us start this section by proving some results concerning the compactness 
properties of the functional ${\mathcal I}_{V}$.
These are  nowadays quite classical in  this context, 
we will follow arguments in \cite{MaiaPellacci, ClappMaia}. 
\begin{lemma}\label{freePS}
Let $\mathcal{N}_V $ be defined in \eqref{defnehari}.
Any sequence $(u_k)$ such that
\[ u_k \in \mathcal{N}_V \quad \text{ and} \quad \mathcal{I}_V(u_k) \to d,\quad
\nabla _{\mathcal  N} \mathcal{I}_V(u_k)\to 0,
\]
satisfies \(\nabla \mathcal{I}_V(u_k)\to 0 \) in $H^{-1}(\R^{N})$
and $(u_{k})$ has a subsequence which is bounded in $H^1(\R^{N})$.
\end{lemma}
\begin{proof}
In order to show that $(u_{k})$ is a free Palais-Smale sequence, 
it is possible to perform the same  argument as in the local case thanks to
the homogeneity power of the Choquard nonlinearity. 
Indeed, the proof of Corollary 3.2 in \cite{clasal1}  can be adapted in a 
straightforward way to the framework without symmetries.
\end{proof}

\begin{lemma}[Splitting Lemma]\label{splitting}
Let $(u_k)$ be a bounded $(PS)_d$ sequence for $\mathcal{I}_V$. Up to a subsequence, there exists a solution $u_0$ of problem \eqref{Choquardeq}, a number $m \in \mathbb{N} \cup \{ 0 \}$, $m$ non trivial solutions $\omega^1, \omega^2, \dots, \omega^m$ to the limit problem \eqref{Choqlimit}, and $m$ sequences of points $(y_k^j) \in \mathbb{R}^N$, $1 \le j \le m$, satisfying
\begin{itemize}
\item[(i)] $\abs{y_k^j} \to \infty$ and $\abs{y_k^j - y_k^i } \to \infty $ if $i \ne j$
\item[(ii)] $u_k - \sum_{i=1}^{m} \omega^i (\cdot - y_k^i) \to u_0$ in $
H^{-1}(\mathbb{R}^N)$
\item[(iii)] $d=\mathcal{I}_V(u_0) + \sum_{i=1}^{m} \mathcal{I}_\infty(\omega^i)$.
\end{itemize}
\end{lemma}
\begin{proof}
Since $(u_k)$ is bounded, there exists $u_0$ such that up to a subsequence $u_k \rightharpoonup u_0$. Then  $\mathcal{I}'_V(u_0)=0$ in $H^{-1}(\R^{N})$. 
Let $u_k^1=u_k-u_0$. By standard arguments and exploiting \eqref{V1}, one has
\[ \int_{\R^N} \abs{\nabla u_k^1}^2 + V_\infty \int_{\R^N} \abs{u_k^1}^2  =\norm{u_k}_V^2 - \norm{u_0}_V^2 + o_k(1). \]
Moreover, by a Brezis-Lieb type lemma (see Lemma 2.4 in \cite{MorozVanSchaftJFA}), one has
\[
\begin{split}
\mathcal{I}_\infty(u_k^1)-\mathcal{I}_V(u_k)+\mathcal{I}_V(u_0) 
=& \int_{\R^N} (I_{\alpha} \ast u_k^{2}) u_k^{2}
- \int_{\R^N} (I_{\alpha} \ast (u_k-u_0)^2) (u_k-u_0)^{2}
\\
& - \int_{\R^N} (I_{\alpha} \ast u_0^2) u_0^{2} + o_k(1) 
\\
=&o_k(1),
\end{split}
\]
where $ o_k(1)\to 0 $ as $k\to+\infty$.
Moreover, as $(u_{k}) $ is a Palais-Smale of ${\mathcal I}_{V}$ at 
level $d$, one has
\[ \mathcal{I}_\infty(u_k^1) \to d - \mathcal{I}_V(u_0). \]
Now, exploiting \cite[Lemma 3.4]{Ackermann}, 
and recalling that ${\mathcal I}_{V}(u_{0})=0$ in $H^{-1}(\R^{N})$, it results, 
for any $\varphi \in H^1$
\begin{multline*}
 o_k(1) \norm{\varphi} \ge \abs{ \langle \mathcal{I}_V'(u_k), \varphi \rangle} \ge \abs{ \langle \mathcal{I}_V'(u_0), \varphi \rangle + \langle \mathcal{I}_\infty'(u_k^1), \varphi\rangle} \\
- \abs{ \int_{\R^N} \left( (I_{\alpha} \ast \abs{u_k^1-u_0}^2) \abs{u_k^1-u_0} - (I_{\alpha} \ast \abs{u_k^1}^2) \abs{u_k^1} - (I_{\alpha} \ast \abs{u_0}^2) \abs{u_0} \right) \varphi} \\
= \abs{ \langle \mathcal{I}_\infty'(u_k^1), \varphi \rangle} - \abs{ \int_{\R^N} \left( (I_{\alpha} \ast \abs{u_k^1-u_0}^2) \abs{u_k^1-u_0} - (I_{\alpha} \ast \abs{u_k^1}^2) \abs{u_k^1} - (I_{\alpha} \ast \abs{u_0}^2) \abs{u_0} \right) \varphi} \\
\ge \abs{ \langle \mathcal{I}_\infty'(u_k^1), \varphi \rangle} -  o_k(1)\norm{\varphi},
\end{multline*} 
Hence
\[ \mathcal{I}_\infty'(u_k^1) \to 0 \text{ in } H^{-1}(\R^{N}). \]
If $u_k^1 \to 0$ strongly, then we are done, choosing $m=0$. 
Otherwise,  by applying Lions lemma \cite{Lions}, there exists $\delta>0$ and a sequence $y_k^1$ such that 
\[ \int_{B_1(y_k^1)} (u_k^1)^2 > \delta. 
\]
 We define $v_k^1(x)=u_k^1(x+y_k^1)$. By boundedness of $(u_k^1)$, there exists $\omega^1$ such that $v_k^1 \rightharpoonup \omega^1$ and $v_k^1 \to \omega^1$ a.e. Therefore, $\omega^1 \ne 0$, and $\abs{y_k^1}\to \infty$. We now show that $\omega^1$ is a solution to \eqref{Choqlimit}. 
Indeed, let $\varphi \in C_c^\infty(\R^{N})$ and set $\varphi_k^1(x)=\varphi(x-y_k^1)$. One has
\[
\mathcal{I}_\infty'(\omega^1)\varphi+o_{k}(1)=\mathcal{I}_\infty'(v_k^1)\varphi=\mathcal{I}_\infty'(u_k^1) \varphi_k^1=o_{k}(1), 
\]
thus $\mathcal{I}_\infty'(\omega^1)=0$ and $\omega^1$ is a solution to \eqref{Choqlimit}. 
Moreover, by \cite[Lemma 2.4]{MorozVanSchaftJFA} and \cite[Lemma 3.4]{Ackermann}, 
\[ \mathcal{I}_\infty(v_k^1)=\mathcal{I}_\infty(v_k^1-\omega^1)+ \mathcal{I}_\infty(\omega^1) + o(1), 
\]
so that
\[ o(1)=\mathcal{I}_\infty' (v_k^1)=\mathcal{I}_\infty'(v_k^1-\omega^1) + o(1). 
\]
Now, we iterate the argument and define  
$u_k^2(x)=u_k^1(x)-\omega^1(x-y_k^1)$. Then, we have
\[\begin{split} 
\mathcal{I}_\infty(u_k^2)&=
\mathcal{I}_\infty(u_k^1)-\mathcal{I}_\infty(\omega^1)+o(1)=\mathcal{I}_V(u_k)-\mathcal{I}_V(u_0)-\mathcal{I}_\infty(\omega^1)+o(1) 
\\
&=d- \mathcal{I}_V(u_0)-\mathcal{I}_\infty(\omega^1)+o(1) 
\end{split}\]
and
\[ \mathcal{I}_\infty'(u_k^2)=\mathcal{I}_\infty'(u_k^1)+o_{k}(1)=o_{k}(1). \]
If $u_k^2 \to 0$ strongly, then we are done, choosing $m=1$. If not, we repeat the argument. After a finite number of steps (as $d$ is finite), we will arrive to a sequence $u_k^{m+1}$ which converges strongly to 0, and the proof is completed. 
\end{proof}

\begin{corollary}\label{PS}
Suppose that  Problem \eqref{Choqlimit} has a unique positive solution. If $c_V$ is not attained, then $c_V \ge c_{\infty}$. Moreover,  ${\mathcal I}_V$ satisfies  the
Palais-Smale condition at every level $d\in (c_{\infty},2c_{\infty})$.
\end{corollary}
\begin{proof}
First of all let us observe that if $d=c_{V}<c_{\infty}$ then, 
conclusion (iii) of Lemma \ref{splitting} and \eqref{eq:nehariinf} imply
\[
d={\mathcal I}_{V}(u_{0})+\sum_{i=1}^{m}{\mathcal I}_{\infty}(\omega^{i})>
{\mathcal I}_{V}(u_{0})+md\geq md
\]
which immediately gives $m=0$, namely $(u_{k})$ strongly converges to
$u_{0}$ so that $c_{V}$ is attained.
\\
On the  other hand, if $d\in (c_{\infty},2c_{\infty})$, one  again exploits conclusion 
(iii) of Lemma \ref{splitting},  recalling that $u_{0}$ is a solution (so that
${\mathcal I}_{V}(u_{0})=0$ if $u_{0}=0$, otherwise ${\mathcal I}_{V}(u_{0})\geq c_{V}$),  and  takes into
account \eqref{defnehari} to obtain
\[
2c_{\infty}>d={\mathcal I}_{V}(u_{0})+ \sum_{i=1}^{m}{\mathcal I}_{\infty}(\omega^{i})\geq mc_{\infty}
\]
because for every $\omega^{i}$, ${\mathcal I_{\infty}}(\omega^{i})\geq c_{\infty}$.
Then $m\leq 1$ and  it is only left to rule out the case $m=1$.
In this case, if  $u_{0}\not \equiv 0$
\begin{equation}\label{comp}
2c_{\infty}>d={\mathcal I}_{V}(u_{0})+ {\mathcal I}_{\infty}(\omega^{1})\geq c_{V}+
 {\mathcal I}_{\infty}(\omega^{1})\geq  c_{\infty}+ {\mathcal I}_{\infty}(\omega^{1})\geq 2c_{\infty}
\end{equation}
which is an evident contradiction,  so that $u_{0}\equiv 0$. 
Reading again \eqref{comp}, we get $
{\mathcal I}_{\infty}(\omega^{1})\in (c_{\infty},2c_{\infty})$. On the other hand
$(\omega^{1})^{\pm}$ both belong to ${\mathcal N}_{\infty}$, so that 
\[
2c_{\infty}>{\mathcal I}_{\infty}(\omega^{1})={\mathcal I}_{\infty}(\omega^{1})^{+}+{\mathcal I}_{\infty}(\omega^{1})^{-}\geq 2c_{\infty},
\]
as a consequence $\omega^{1}$ does not change sign and it is, up to a translation, the unique positive solution of Problem \eqref{Choqlimit}
and ${\mathcal I}_{\infty}(\omega^{1})=c_{\infty}$.
Finally, writing again (iii) of Lemma \ref{splitting} with
$u_{0}\equiv0$ and $m=1$ gives
\[
d= {\mathcal I}_{\infty}(\omega^{1})= c_{\infty} 
\]
which is again a contradiction, yielding $m=0$.
\end{proof}

Recall \eqref{defeps} and  define for every $\lambda\in [0,1]$
\begin{equation}\label{def:chi} 
\chi_{\lambda,R}= \lambda \omega_{1, R}+(1-\lambda)\omega_{2,R}, \qquad \text{where 
$\omega_{i, R}$ is defined in \eqref{omegaR}}.
\end{equation}

\begin{lemma}\label{radialproj2}
For every $u\in H^{1}(\R^{N})\setminus \{0\}$ the   number
\begin{equation}\label{defT}
\overline{T}:=\left[\frac{\|u\|_{V}^{2}}{\int_{\R^{N}}(I_{\alpha}\ast u^{2})u^{2}}\right]^{1/2}
\end{equation}
is the unique positive one such that $\overline{T} u\in \mathcal{N}_V$ and
the map $T:  H^{1}(\R^{N})\setminus \{0\}\mapsto \R^{+}$ is continuous.
In addition, having defined 
\(
T_{\lambda,R}:=\overline{T}(\chi_{\lambda,R}),
\)
 there exists $R_0>0$ and $T_{0}>0$ such that 
\begin{equation}\label{Tbound}
T_{\lambda,R} \leq T_0, \qquad \text{ for every $R \ge R_0$, and  $\lambda\in [0,1]$.}
\end{equation}
\end{lemma}
\begin{proof}
The fact that $\overline{T}$ given in \eqref{defT} is the unique positive number
such that $\overline{T}u\in \mathcal{N}_V$ comes from the characterization
\[
{\mathcal I}_{V}(\overline{T}u)=\max_{t>0}{\mathcal I}_{V}(tu),
\]
and the continuity is a direct consequence of \eqref{defT}.
In order to prove \eqref{Tbound} let us note that Lemma \ref{estimatesV} and 
Theorem \ref{thm:decay} imply
\[
\begin{split}
\|\chi_{\lambda,R}\|_{V}^{2}=&
(\lambda^{2}+(1-\lambda)^{2})\|\nabla \omega\|_{2}^{2}+
2\lambda(1-\lambda)\int_{\R^{N}}\nabla \omega_{1,R}\nabla \omega_{2,R}
\\
&+\lambda^{2}\int_{\R^{N}}V(x+Rz_{1})\omega^{2}+
(1-\lambda)^{2}\int_{\R^{N}}V(x+Rz_{2})\omega^{2}+2\lambda(1-\lambda)\int_{\R^{N}}
V(x)\omega_{1,R}\omega_{2,R}
\\
&=\left[\lambda^{2}+(1-\lambda)^{2}\right]\left[
\|\nabla \omega\|_{2}^{2}+\int_{\R^{N}}V_{\infty}\omega^{2}\right]+o_{R}(1),
\end{split}
\]
where $o_{R}(1)$ is a quantity tending to zero as $R\to +\infty$. In addition,
one has
\[
\begin{split}
\int_{\R^{N}}(I_{\alpha}\ast \chi^{2}_{\lambda,R})\chi^{2}_{\lambda,R}\geq  &
\left[\lambda^{4}+(1-\lambda)^{4}\right]\int_{\R^{N}}(I_{\alpha}\ast \omega^{2})\omega^{2}
\end{split}
\]
so that
\[
\begin{split}
T^{2}_{\lambda,R}=\frac{\lambda^{2}+(1-\lambda)^{2}}
{\lambda^{4}+(1-\lambda)^{4}}
\frac{\|\nabla \omega\|_{2}^{2}+\int_{\R^{N}}V_{\infty}\omega^{2}+o_{R}(1)}{ 
\int_{\R^{N}}(I_{\alpha}\ast \omega^{2})\omega^{2}}
\leq 4+o_{R}(1).
\end{split}
\]
\end{proof}
\begin{proposition}\label{estimates2}
Assume \eqref{V1}, \eqref{Vdecaygen}, \eqref{ipo:sigma}. 
Then, there exists $R_1>0$  and for each $R > R_1$ there exists  $\eta_{R}>0$    such that 
\begin{equation}\label{eq:livellopalla}
\mathcal{I}_V(T_{\lambda,R } \chi_{\lambda,R}) \le 2 c_{\infty} - \eta_{R} 
\end{equation}
for all $\lambda \in [0, 1]$ and all $z \in \partial B_2(z_1)$. 
Moreover, for any $\delta>0$ there exists $R_{2}>0$ such that 
\begin{equation}\label{eq:livellobordo}
\mathcal{I}_V(T_{0,R } \chi_{0, R})=
\mathcal{I}_V(T_{0,R } \omega_{2, R}) < c_{\infty} + \delta 
\end{equation}
for every $z \in \partial B_2(z_0)$ and $R > R_2$. 
 In particular, $c_V \le c_{\infty}$. 
\end{proposition}
\begin{proof}
Let us first  note that \eqref{defT} yields
\begin{equation}\label{eq:iv}
\mathcal{I}_V(T_{\lambda,R } \chi_{\lambda,R})=\frac14 T_{\lambda,R }^{2}\|\chi_{\lambda,R}\|_{V}^{2}.
\end{equation}
Repeating the argument in  the proof of Lemma \ref{radialproj2}, taking into 
account Lemma \ref{estimatesV}, \eqref{defeps} and the fact that $\omega$ is 
a solution of \eqref{Choqlimit} we get
\begin{equation}\label{eq:ablambda}
\|\chi_{\lambda,R}\|_{V}^{2}=
(s^{2}+t^{2})\|\omega\|^{2}+2st\eps_{R}+o(\eps_{R}),
\quad \text{where $s=\lambda$,  $t=(1-\lambda)$, }
\end{equation}
On the other hand, Proposition \ref{prop:somma} and Lemma \ref{estimatesV}  yield
\[\begin{split}
\int_{\R^{N}}(I_{\alpha}\ast\chi^{2}_{\lambda,R})\chi^{2}_{\lambda,R}
&\geq 
(s^{4}+t^{4})
\int_{\R^{N}}(I_{\alpha}\ast\omega^{2})\omega^{2} +4st(s^{2}+t^{2}) \eps_{R}
\\
&=(s^{4}+t^{4})
\|\omega\|^{2} +4st(s^{2}+t^{2}) \eps_{R}.
\end{split}\]
Using these information in  \eqref{defT} 
and taking into account the expansion 
$(a+bt)^{-1}=\frac1a-\frac{b}{a^{2}}t+o(t)$ one gets
\[\begin{split}
T^{2}&
\leq \left\{(s^{2}+t^{2})\|\omega\|^{2}+2st\eps_{R}+o(\eps_{R})\right\}
\left\{\frac1{(s^{4}+t^{4})\|\omega\|^{2}}-\frac{4st(s^{2}+t^{2}) }{(s^{4}+t^{4})^{2}\|\omega\|_{V}^{4}}\eps_{R}+o(\eps_{R})
\right\}
\\
&=\frac{s^{2}+t^{2}}{(s^{4}+t^{4})}+\frac{2st}{(s^{4}+t^{4})\|\omega\|^{2}}\eps_{R}\left\{1-2\frac{(s^{2}+t^{2})^{2}}{s^{4}+t^{4}}
\right\}+o(\eps_{R}).
\end{split}
\]
When using this inequality in \eqref{eq:iv} one obtains, thanks to \eqref{eq:ablambda},
\begin{equation}\label{final:est2}
\begin{split}
\mathcal{I}_V(T_{\lambda,R } \chi_{\lambda,R})
&\leq
\frac14\|\omega\|_{V}^{2}\frac{(s^{2}+t^{2})^{2}}{s^{4}+t^{4}} + 
\frac{st(s^{2}+t^{2})}{(s^{4}+t^{4})}\eps_{R}\left\{1-\frac{(s^{2}+t^{2})^{2}}{(s^{4}+t^{4})}
\right\}
\\
&=c_{\infty}\frac{(s^{2}+t^{2})^{2}}{s^{4}+t^{4}}-\frac{2s^{3}t^{3}(s^{2}+t^{2})}{(s^{4}+t^{4})^{2}}\eps_{R}+o(\eps_{R})
\end{split}\end{equation}
Let now $\delta$ be a positive constant less than $1/2$.
If $\lambda\in [\frac12-\delta,\frac12+\delta]$ one observes that
\[
\frac{2s^{3}t^{3}(s^{2}+t^{2})}{(s^{4}+t^{4})^{2}}\geq \mu_{\delta}>0,\quad 
\frac{(s^{2}+t^{2})^{2}}{s^{4}+t^{4}}\leq 2,
\]
so that \eqref{final:est2} becomes
\[
\mathcal{I}_V(T_{\lambda,R } \chi_{\lambda,R}) \leq
2 c_{\infty}- \mu_{\delta}\eps_{R} +o(\eps_{R})\leq2c_{\infty}-\eta_{R}.
\]
In the case $|\lambda-\frac12|>\delta$, there exists $\sigma_{\delta}$ such that
that 
\[
\frac{(s^{2}+t^{2})^{2}}{s^{4}+t^{4}}\leq 2-\sigma_{\delta}<2
\]
so that
\eqref{final:est2} becomes 
\[
\mathcal{I}_V(T_{\lambda,R } \chi_{\lambda,R}) \leq
(2-\sigma_{\delta})c_{\infty}-\frac{2s^{3}t^{3}(s^{2}+t^{2})}{(s^{4}+t^{4})^{2}}\eps_{R} +o(\eps_{R})<2c_{\infty}.
\]
Then, also in this case we get \eqref{eq:livellopalla}.
In order to show \eqref{eq:livellobordo} we repeat the same argument with $\lambda=0$ arriving at \eqref{final:est2} which now reads as follows
(recalling that $s=0$, $t=1$)
\[
\mathcal{I}_V(\overline{T}(\omega_{2, R}) \omega_{2, R})
 \leq  c_{\infty}+o(\eps_{R}).
\]
which  immediately yields \eqref{eq:livellobordo}.
\end{proof}

From now on we will make use of a barycenter map, whose definition and properties we briefly recall for the sake of completeness.
For more details see \cite{CeramiPassaseo}, \cite{bawe}, \cite{amceru}, \cite{MaiaPellacci}.
For every $u\in H^{1}(\R^{N})\setminus \{0\}$, the 
following maps are well defined
\[
\mu(u)(x):=\frac{1}{|B_{1}(x)|}\int_{B_{1}(x)}|u(y)|dy,\quad
\text{$\mu(u)\in L^{\infty}\cap C^{0}(0,+\infty)$ },
\]
$$
\hat{u}(x):=\left[\mu(u)(x)-\frac{\|\mu(u)\|_{\infty}}2\right]^{+},
\quad \hat{u}\in C_{0}(\R^{N}).
$$
Then, the barycenter  of a function $u\in H^{1}(\R^{N})\setminus \{0\}$  defined by
$$
\beta(u)=\frac{1}{\|\hat{u}\|_{1}}\intr x\hat{u}(x)dx
$$
 is a continuous function enjoying the following properties 
\begin{align}
\label{btras} \beta(u(\cdot-y))&=\beta(u)+y\quad \forall\,y\in \R^{N},
\\
\label{bmolt}\beta(Tu)&=\beta(u)\quad \forall\,T>0.
\end{align}
Note that $\beta(u)=0$ if $u$ is radial. 

\begin{lemma}\label{cv=cinf}
If $c_V$ is not attained then $c_V = c_{\infty}$ and there exists $\delta>0$ such that 
\[ \beta(u) \ne z_1 \quad \forall u \in \mathcal{N}_V \cap \mathcal{I}_V^{c_{\infty}+\delta} \]
where $\mathcal{I}_V^c= \{ u : \mathcal{I}_V(u) \le c \}$. 
\end{lemma}
\begin{proof}
By Corollary \ref{PS} and Proposition \ref{estimates2} if $c_V$ is not attained, then $c_V=c_{\infty}$. 
Let us assume by contradiction that for each $k \in \mathbb{N}$ there exists $u_k \in \mathcal{N}_V$ such that 
\[
\mathcal{I}_V(u_k) < c_V + \frac{1}{k},\quad  \text{and} \quad \beta(u_k)=z_1.
\]
 By Ekeland's variational principle (see \cite{Willem}) there exists a constrained Palais-Smale sequence,  called  $(v_k)$, at level $c_V$ for $\mathcal{I}_V$ on $\mathcal{N}_V$ and such that $\|v_{k}-u_{k}\|_{V}\to 0$, so that $\beta(v_{k})\to z_{1}$.
  By Lemma \ref{freePS}  $(v_k)$ is (up to a subsequence) a  bounded Palais Smale sequence for $\mathcal{I}_V$ at level $c_V$  in $H^{1}(\R^{N})$. 
Since $c_V$ is not attained, we conclude by Lemma \ref{splitting} that there exists 
a sequence $(z_k)$ such that $\abs{z_k} \to \infty$ such that 
$\norm{v_k - \omega(\cdot - z_k)} \to 0$. We set $ w_k(x)=v_k(x+z_k)$ and 
from \eqref{btras}, it follows that 
\[ 
z_1 - z_k =\beta(v_k)-z_k+o_{k}(1)=\beta (w_k)+o_{k}(1) \to \beta(\omega)=0, 
\]
which is a contradiction. 
\end{proof}
We are finally in the position to prove our main results.
\begin{proof}[Proof of Theorem \ref{thm:abstract}]
If $c_V$ is attained at some $u \in \mathcal{N}_V$, taking into account that 
$\mathcal{N}_V$ is a natural constraint for $\mathcal{I}_{V}$ it turns out that
$u$ is a nontrivial solution of \eqref{Choquardeq}. Let us assume that $c_V$ is not attained. Then by Lemma \ref{cv=cinf} $c_V=c_{\infty}$. Then, we are going to show that $\mathcal{I}_V$ has a critical value in $(c_{\infty}, 2c_{\infty})$. 
By Lemma \ref{cv=cinf} we may find a $\delta >0$ sufficiently small such that 
\[ 
\beta(u) \ne z_1 \quad \forall u \in \mathcal{N}_V \cap \mathcal{I}_V^{c_{\infty}+\delta}, \]
where $\mathcal{I}_V^{b}:=\{u\in H^{1}(\R^{N}) : \mathcal{I}_V (u)\leq b\}$.

Moreover, thanks to  Proposition \ref{estimates2}  we can choose $\eta > 0$ 
sufficiently small and $R>0$ such that 
\[
\mathcal{I}_V(T_{\lambda,R} \chi_{\lambda,R}) \le \begin{cases}
2c_{\infty} - \eta & \text{ for all } \lambda \in [0, 1] \text{ and all } z_{2} \in \partial B_2(z_1) \\
c_{\infty}+\delta & \text{ for } \lambda=0  \text{ and all } z_{2} \in \partial B_2(z_1).
\end{cases}
\]
Let us define $\psi : \overline{B_2(z_1)} \to \mathcal{N}_V \cap \mathcal{I}_V^{2c_{\infty} -\eta}$ by
\[ \psi(\lambda z_1 + (1-\lambda)z_{2}) = T_{\lambda,R} \chi_{\lambda,R}, \text{ with } \lambda \in [0, 1], z_{2} \in \partial B_2(z_1). \]
Let us assume by contradiction that $\mathcal{I}_V$ does not have a critical value in $(c_{\infty}, 2c_{\infty})$. 
Thus,  one can define a continuous deformation (see Lemma 5.15 in \cite{Willem})
\[ 
\rho: \mathcal{N}_V \cap \mathcal{I}_V^{2c_{\infty} -\eta} \mapsto \mathcal{N}_V \cap \mathcal{I}_V^{c_{\infty} +\delta} 
\]
such that $\rho(u)=u$ for all $u \in \mathcal{N}_V \cap \mathcal{I}_V^{c_{\infty} +\delta}$. 
Then the function $h: \overline{B_2(z_1)} \to \partial B_2(z_1)$ given by 
\[ 
h(x)=2 \left( \frac{(\beta \circ \rho \circ \psi)(x) - z_1}{\abs{(
\beta \circ \rho \circ \psi)(x) - z_1}} \right) + z_1 
\]
is well defined and continuous. Moreover, if $z_{2} \in \partial B_2(z_1)$, 
then 
\[
\psi(z_{2})=T_{0,R} \chi_{0,R} =T_{0,R} \omega_{2,R} \in \mathcal{N}_V \cap \mathcal{I}_V^{c_{\infty} +\delta},
\] 
and $(\beta \circ \rho \circ \psi)(z_{2})=\beta (T_{0,R}\omega_{2,R})=z_{2}$. Therefore, $h(z_{2})=z_{2}$ for every $z_{2} \in \partial B_2(z_1)$. Since such a map does not exists, $\mathcal{I}_V$ must have a critical point $u$.  
Noting that $u^{\pm}$ both belong to $\mathcal{N}_V $ so that
$\mathcal{I}_V(u^{\pm})\geq \mathcal{I}_V(u)$ and at the same time
\[
\mathcal{I}_V(u^{+})+\mathcal{I}_V(u^{-}) =\mathcal{I}_V(u)= c_{V}\in(c_{\infty},2c_{\infty})
\]
allows us to conclude that $u$ can be chosen nonnegative and by the Maximum Principle $u$ is positive.
\end{proof}
\begin{proof}[Proof  of Theorem \ref{thm:uniq}]
Theorem \ref{thm:uniq} immediately follows once one  notices that hypothesis
\eqref{ipo:sigma} reduces to  \eqref{iposigma} when $\alpha=2$ and $N=3,4,5$.
\end{proof}
	


\begin{thebibliography}{9}
\bibitem{Ackermann} N. Ackermann, \emph{On a periodic Schr\"odinger 
equation with nonlocal superlinear part}, \textit{Math. Z.}, \textbf{248} 
(2004), 413--443.
\bibitem{Alves} C.O. Alves, A.B. N\'obrega, M. Yang
\emph{Multi-bump solutions for Choquard equation with deepening potential well}
\textit{ Calc. Var. Partial Differential Equations}  \textbf{55}, (2016) 55--48.

\bibitem{AmbrosettiColoradoRuiz} A. Ambrosetti, E. Colorado, D. Ruiz, 
\emph{Multi-bump solitons to linearly coupled systems of nonlinear Schr\"odinger equations}, \textit{Calc. Var.} \textbf{30} (2007), 85--112. 

 \bibitem{amceru}  A. Ambrosetti, G. Cerami and D. Ruiz, 
\emph{Solitons of linearly coupled systems of semilinear 
non-autonomous equations on $R^{n}$},
J. Funct. Anal. \textbf{254}, no.11, (2008), 2816--2845.

\bibitem{bawe} T. Bartsch, T. Weth, 
\emph{Three nodal solutions of singularly elliptic equations on
domains without topology},
Ann. I. H. Poincar\'e Anal. Non Lin\'eaire \textbf{22},
no. 3, (2005),  259--281.

\bibitem{BenciCerami} V. Benci, G. Cerami, \emph{Positive solutions of some 
nonlinear elliptic problems in exterior domains}, Arch. Rational
Mech. Anal. \textbf{99}, no. 4, (1987), 283--300.

\bibitem{CeramiPassaseo} G. Cerami, D. Passaseo, Existence and multiplicity results for semilinear elliptic Dirichlet problems in exterior domains, \textit{Nonlinear Anal.} \textbf{24} (1995), 1533--1547.

\bibitem{CingolaniClappSecchi} S. Cingolani, M. Clapp, S. Secchi, Multiple solutions to a magnetic nonlinear Choquard equation \textit{Z. Angew. Math. Phys.} \textbf{63} (2012), 233--248.

\bibitem{ClappMaia} M. Clapp, L. Maia, A positive bound state for an asymptotically linear or superlinear
Schr\"odinger equation, \textit{J. Differential Equations} \textbf{260} (2016), 3173--3192.


\bibitem{clasal} M. Clapp, D. Salazar, Positive and sign changing solutions to a 
nonlinear Choquard equation, \textit{J. Math. Anal. Appl.} \textbf{407} (2013), 
1--15. 

\bibitem{clasal1}  M. Clapp, D. Salazar, Multiple Sign Changing Solutions of 
Nonlinear Elliptic Problems in Exterior Domains, \textit{Advanced Nonlinear 
Studies}, \textbf{12}, (2012), 427--443.

\bibitem{GhimentiMorozVanSchaft} M. Ghimenti, V. Moroz, J. Van Schaftingen, Least action nodal solutions for the quadratic Choquard equation, \textit{Proc. Amer. Math. Soc.} \textbf{145} (2017), 737--747. 

\bibitem{GhimentiVanSchaft} M. Ghimenti, J. Van Schaftingen, 
Nodal solutions for the Choquard equation, \textit{ J. Funct. Anal.} \textbf{271} (2016), 107--135. 

\bibitem{Lieb} E.H. Lieb, Existence and uniqueness of the minimizing solutions of Choquard's nonlinear equation, \textit{Studies in Appl. Math.} \textbf{57} (1976/77), 93--105.

\bibitem{Lions} P.L. Lions, The Choquard equation and related questions   \textit{Nonlinear Analysis} \textbf{4} (1980), 1063--1073.
 
\bibitem{Lions2} P.L. Lions, The concentration-compactness principle in the 
calculus of variations. The locally compact case. I. \textit{Ann. Inst. H. Poincaré 
Anal. Non Linéaire} \textbf{1} (1984), 109--145. 

\bibitem{MaZhao} L. Ma, L. Zhao, Classification of positive solitary solutions of the 
nonlinear Choquard equation, \textit{Arch. Ration. Mech. Anal.} \textbf{195} 
(2010), 455--467. 

\bibitem{MaiaPellacci} L. Maia, B. Pellacci, Positive solutions for asymptotically 
linear problems in exterior domains, \textit{Ann. Mat. Pura Appl. (4)} \textbf{196} 
(2017), 1399--1430.

\bibitem{MaPeSc} L. Maia, B. Pellacci, D. Schiera, {\em Symmetric Positive solutions to nonlinear Choquard Equations}. Preprint.

\bibitem{MorozVanSchaftJFPTA} V. Moroz, J. Van Schaftingen, A guide to the Choquard equation, \textit{Journal of Fixed Point Theory and Applications} \textbf{19} (2017), 773--813.

\bibitem{MorozVanSchaftJDE} V. Moroz, J. Van Schaftingen, Nonexistence and optimal decay of supersolutions to Choquard equations in exterior domains, \textit{J. Differential Equations} \textbf{254} (2013), 3089--3145.

\bibitem{MorozVanSchaftJFA} V. Moroz, J. Van Schaftingen, Groundstates of 
nonlinear Choquard equations: existence, qualitative properties and decay 
asymptotics, \textit{J. Funct. Anal.} \textbf{265} (2013), 153--184.

\bibitem{Struwe} 
M. Struwe, A global compactness result for elliptic boundary value problems involving limiting nonlinearities, \textit{Math Z.} \textbf{187} (1984), 511--517.

\bibitem{VanShaftXia} J. Van Schaftingen, J. Xia, Choquard equations under
confining external potentials, \textit{Nonlinear Differ. Equ. Appl. } 
 (2017), 1--24. DOI 10.1007/s00030-016-0424-8
 
\bibitem{WangQuXiao} J. Wang, M. Qu, L. Xiao, Existence of positive solutions to the nonlinear Choquard equation with competing potentials, \textit{Electr. J. Differ. Equ.} \textbf{63} (2018), 1--21. 
\bibitem{Willem} M. Willem, \textit{Minimax theorems}, Progress in Nonlinear Differential Equations and their Applications, 24, Birkh\"auser Boston, Inc., Boston, MA, 1996, x+162 pp.
\bibitem{Xiang} C. L. Xiang, Uniqueness and nondegeneracy of ground states for Choquard equations in three dimensions, \textit{Calc. Var. Partial Differential Equations} \textbf{55} (2016), Art. 134, 25 pp. 
\bibitem{wangyi}  T. Wang, T. Yi, \emph{Uniqueness of positive solutions of
the Choquard type equations}, \textit{Applicable Analysis}, \textbf{96} (2017),
409--417.



\end{thebibliography}
\end{document}